\documentclass[a4paper]{amsart}

\usepackage{amsfonts}
\usepackage{amsfonts}
\usepackage{amssymb}
\usepackage{nicefrac}
\usepackage{enumerate}
\usepackage{url}
\usepackage{eufrak}

\newcommand{\R}{\mathbb{R}}
\newcommand{\N}{\mathbb{N}}
\newcommand{\T}{\mathbb{T}_m}
\renewcommand{\S}{\mathcal{S}}

\newcommand{\Lloc}{L^{\infty}_{loc}([0, +\infty); L^\infty (\T, \R) )}
\DeclareMathOperator{\med}{\,median\,}

\newtheorem{teo}{Theorem}[section]

\newtheorem{co}[teo]{Corollary}
\newtheorem{pro}[teo]{Proposition}

\theoremstyle{remark}
\newtheorem{re}[teo]{Remark}
\theoremstyle{example}
\newtheorem{ex}[teo]{Example}

\theoremstyle{definition}

\title[Evolution equations on trees]{Existence, uniqueness and decay rates for evolution equations on trees}
\author[L. M. Del Pezzo, C. A. Mosquera and J. D. Rossi]
{Leandro M. Del Pezzo, Carolina A. Mosquera and Julio D. Rossi}

\address{Leandro M. Del Pezzo and Carolina A. Mosquera
\hfill\break\indent
CONICET and Departamento  de Matem{\'a}tica, FCEyN,
Universidad de Buenos Aires,
\hfill\break\indent Pabellon I, Ciudad Universitaria (1428),
Buenos Aires, Argentina.}

\email{{\tt ldpezzo@dm.uba.ar, mosquera@dm.uba.ar}}

\address{Julio D. Rossi \hfill\break\indent
Departamento  de An{\'a}lisis Matem{\'a}tico, Universidad de Alicante,
\hfill\break\indent Ap. correo 99, 03080, Alicante, SPAIN.
}

\email{{\tt julio.rossi@ua.es}}

\date{\today}

\thanks{
Leandro M. Del Pezzo was partially supported by UBACyT 20020110300067
and CONICET PIP 5478/1438  (Argentina) ,
Carolina A. Mosquera was partially supported by UBACyT 20020100100638, PICT 0436 and
CONICET PIP  112 200201 00398 (Argentina)
and Julio D. Rossi  was partially supported by MTM2011-27998,
(Spain) }

\begin{document}
\begin{abstract}
We study evolution equations governed by an averaging operator on a directed tree, showing existence and uniqueness of solutions. In addition we find conditions of the initial condition that allows us to find the asymptotic decay rate of the solutions as $t\to \infty$. It turns out that this decay rate is not uniform, it strongly depends on how the initial condition goes to zero as one goes down in the tree.
\end{abstract}

\maketitle

\section{Introduction}
Let $\T$ be a directed tree with $m$-branching, we denote by $x$ the vertices of the tree. Given
a function $f:\T\to\mathbb{R}$,
in this work we study the following Cauchy problem
\begin{equation}
	\label{CP}
	\begin{cases}
		u_t(x,t)-\Delta_F u (x,t) = 0 & \mbox{ in }
		\T\times(0,+\infty),\\
		u(x,0)=f(x) & \mbox{ in } \T,
	\end{cases}
\end{equation}
where
\[
\Delta_F u (x,t)= F(u((x,0),t),\dots,u((x,m-1),t))-u(x,t),
\]
 being $F$ an averaging operator, see the precise definition in Section \ref{previa}. The simplest linear example of an averaging operator is
the usual average
\[
F(x_1,\dots,x_m)=\dfrac{1}m\displaystyle\sum_{j=1}^m x_j
\]
but we can include nonlinear functions as
\[
F(x_1,\dots,x_m)=\dfrac{\alpha}2
			\left(\displaystyle\max_{1\le j\le m}
			\{x_j\}+\displaystyle\min_{1\le j\le m}\{x_j\}
			\right) + \dfrac{1-\alpha}m\displaystyle\sum_{j=1}^m x_j,
\]
 with $0<\alpha<1$.

We can see that $u$ is a solution of \eqref{CP} if and
only if it is a solution of the integral equation
\begin{equation}\label{EI}
u(x,t)=K_fu(x,t),
\end{equation}
where
\[
K_fu(x,t):=\int_{0}^t
e^{z-t}F(u((x,0),z),\dots,u((x,m-1),z))\, dz + e^{-t} f(x).
\]

We first prove existence and uniqueness of a locally bounded global in time solution using a fixed point argument for $K_f$.

\begin{teo}\label{TEU}
Let $f\in L^\infty(\T,\R).$ Then
there exists a unique solution $u$ in
\[
\Lloc:=  \left\{
v\in L^\infty(\T\times[0,T],\R) \ \forall T>0\right\}
\]
of \eqref{CP}.
\end{teo}

In addition, a comparison principle holds.

\begin{teo}\label{TCOMP}
Let $F$ be an averaging operator, $f, g\in L^\infty(\T,\R)$ such that $f\le g$ in $\T,$
and $u,v\in\Lloc$ such that
\begin{equation}\label{des}
u(x,t)\le K_fu(x,t) \quad\mbox{ and }\quad v(x,t)\ge K_gv(x,t)
\end{equation}
for all $(x,t)\in\T\times[0,+\infty)$. Then  $u\le v$ in $\T\times[0,+\infty).$
\end{teo}

Once we have established existence and uniqueness of global in time solutions a natural question is to look for its
asymptotic behaviour as $t\to \infty$. We find conditions on the initial condition $f$ (that involve the speed at which they go to zero as one goes down in the tree) that guarantee that solutions go to zero as $t\to \infty$. Under these conditions we can find bounds for the decay rate. Surprisingly the decay rate for solutions to \eqref{CP} is not uniform. It strongly depends on the decay of the initial condition $f$. For example, for initial conditions with finite support (only a finite number of vertices have non-zero values) we find a decay of the form $t^{\mu} e^{-t}$ (here $\mu$ depends on the size of the support of $f$), while for data without finite support we find a decay of the form $e^{-\lambda t}$ (with $0<\lambda<1$ depending on the decay of $f$).
This is the content of our next results whose proof rely mostly on comparison arguments. For the statements we need to introduce the following notations.

Let $f\in L^{\infty}(\T, \R).$ We will say that $f$
has finite support if there exists $n\in\N_0$ such that
$f(x)=0$ for all $x\in\T$ with $l(x)\ge n,$ where $l(x)$ denotes
the level of $x.$ We also define
\[
\textit{a}(f):= \min\{j\in\N_0\colon f(x)=0,\, \forall x\in\T \text{ with }\textit{l}(x) \ge j\},
\]
and
\[
	\mu(f):=
	\begin{cases}
		\textit{a}(f)-1 & \mbox{ if } \textit{a}(f)>0,\\
		0 & \mbox{ if } \textit{a}(f)=0.
	\end{cases}
\]

\begin{teo}\label{TEST}
Let $F$ be an averaging operator and $f\in L^\infty(\T,\R)$ with
finite support. If $u\in \Lloc$ is the solution of \eqref{CP} with initial condition $f,$ then
\begin{equation}\label{estimacion.2}
	\max_{x\in\T} |u(x,t)| \le
	\frac{t^{\mu(f)}e^{-t}}{\mu(f)!}\|f\|_{L^\infty(\T,\R)},
\end{equation}
for $t$ large enough.
\end{teo}
The above bound is optimal, see Remark \ref{cota.opt}.

\medskip

For $f$ that are not finitely supported we have the following result.

\begin{teo}\label{RiBer.es.la.B}
Let $F$ be an averaging operator and $f\in L^{\infty}(\T, \R)$ such that there exist $\lambda\in(0,1)$ and $k\in\R_{>0}$ such that
\[
	|f(x)|\le k(1-\lambda)^{l(x)} \quad \forall x\in\T.
\]
If $u\in \Lloc$ is the solution of \eqref{CP} with initial condition $f,$
then
\[
	\max_{x\in\T}|u(x,t)|\le ke^{-\lambda t} \quad \forall t\in\R.
\]
\end{teo}

Again this bound is optimal, see Proposition \ref{RiBer.es.el.mejor.del.mundo}.

In the next result we show that we can construct a solution with quite different bahaviours at $\emptyset$, the first node of our tree.

\begin{teo} \label{todo.vale}
Let $F$ be an averaging operator and $a_0(t)\in C^{\infty}([0,\infty), \R)$, then there is a solution of
$u_t(x,t)-\Delta_F u (x,t) = 0$ in $
		\T\times(0,+\infty)$,
such that
\[
u(\emptyset,t) = a_0 (t) \quad \forall t\in\R.
\]
\end{teo}

Let us end the introduction with a brief comment on previous bibliography that concerns
mostly the stationary problem.
For nonlinear mean values on a finite graph we refer to \cite{Ober} and references
therein. For equations on trees like the ones considered here, see \cite{ary,KLW,KW}
and \cite{s-tree,s-tree1}, where for the stationary problem it is proved the existence and
uniqueness of a solution using game theory. See also \cite{dpmr,dpmr2} where the
authors study the unique continuation and find some estimates for the harmonic
measure on trees. Here we use ideas from these references.

The time dependent diffusion equations on simple, connected, undirected graphs, have been used to model diffusion processes, such as, modeling energy
flows through a network or vibration of molecules, \cite{ChB,CM}.

In the case when $F$ is the usual average, it is possible to construct a fundamental
solution for \eqref{CP} on  infinite, locally finite,
connected graphs. See \cite{Web,Woj} and the references therein.

This paper is a natural extension of the previously mentioned references since here we
deal with the evolution problem associated to an averaging operator on a tree that is a
directed graph.

\medskip

This paper is organized as follows: in Section \ref{previa} we collect some preliminaries; in Section \ref{sect-exis-uni} we deal with existence and uniqueness of solutions and prove Theorem \ref{TEU} and Theorem \ref{TCOMP}; in Section \ref{sect-decay} we prove our results concerning the decay of solutions as $t\to \infty$ proving Theorem \ref{TEST}, Theorem \ref{RiBer.es.la.B} and Theorem \ref{todo.vale}.

\section{Preliminaries}\label{previa}

We begin with a review of the basic results that will be needed in subsequent sections.
The known results are generally stated without proofs, but we provide references where
the proofs can be found. Also, we introduce some of our notational conventions.

\subsection{Directed Tree}
Let $m\in\mathbb{N}_{>2}$. In this work we consider a directed
tree $\T$ with regular $m-$branching, that is, $\T$ consists of
the empty set $\emptyset$ and all finite  sequences
$(a_1,a_2,\dots,a_k)$ with $k\in\N,$ whose coordinates $a_i$ are
chosen from $\{0,1,\dots,m-1\}.$ The elements in $\T$ are called
vertices. Each vertex $x$ has $m$ successors, obtained by adding
another coordinate. We will
denote by $\S(x)$ the set of successors of the vertex $x.$ A
vertex $x\in\T$ is called an $n-$level vertex ($n\in\mathbb{N}$) if
$x=(a_1,a_2,\dots,a_n),$ and we will denote by $l(x)$ the level
of vertex $x.$ The set of all $n-$level vertices is
denoted by $\T^n.$

\medskip

A branch of $\T$ is an infinite sequence of vertices, each followed by its immediate successor.
The collection of all branches forms the boundary of $\T$,  denoted by $\partial\T$.

\medskip

We now define a metric on $\T\cup \partial\T.$ The distance
between two sequences (finite or infinite) $\pi=(a_1,\dots,
a_k,\dots)$ and $\pi'=(a_1',\dots, a_k',\dots)$ is $m^{-K+1}$ when
$K$ is the first index $k$ such that $a_k\neq a_k';$ but when
$\pi=(a_1,\dots, a_K)$ and $\pi'=(a_1,\dots, a_K,
a_{K+1}',\dots),$ the distance is $m^{-K}.$  Hausdorff measure and
Hausdorff dimension are defined using this metric. We have
that $\T$ and $\partial\T$ have diameter one and $\partial\T$ has
Hausdorff dimension one. Now, we observe that the mapping
$\psi:\partial\T\to[0,1]$ defined as
\[
\psi(\pi):=\sum_{k=1}^{+\infty} \frac{a_k}{m^{k}}
\]
is surjective, where $\pi=(a_1,\dots, a_k,\dots)\in\partial\T$ and
$a_k\in\{0,1,\dots,m-1\}$ for all $k\in\mathbb{N}.$ Whenever
$x=(a_1,\dots,a_k)$ is a vertex, we set
\[
 \psi(x):=\psi(a_1,\dots,a_k,0,\dots,0,\dots).
\]
We can also associate to a vertex $x$ an
interval $I_x$ of length $\frac{1}{m^k}$ as follows
\[
 I_x:=\left[\psi(x),\psi(x)+\frac1{m^k}\right].
\]
Observe that for all $x\in \T$, $I_x \cap \partial\T$ is the
subset of $\partial\T$ consisting of all branches that start at
$x$.
With an abuse of notation, we will write $\pi=(x_1,\dots,x_k,\dots)$
instead of $\pi=(a_1,\dots,a_k,\dots)$ where $x_1=a_1$ and
$x_k=(a_1,\dots,a_k)\in\S(x_{k-1})$ for all $k\in\N_{\ge2}.$

Finally we will denote by $\T^x$ the set of the vertices $y\in\T$
such that $I_y\subset I_x.$

\begin{ex}
{\rm  Let $\kappa\in\mathbb{N}$ be at least 3.
  A $\nicefrac{1}{\kappa}-$Cantor set, that we denote by
  $C_{\nicefrac{1}{\kappa}}$,
is the set of all $x\in[0,1]$ that have a base $\kappa$ expansion
without the digit $1$,
  that is $x=\sum a_j\kappa^{-j}$ with
  $a_j\in\{0,1,\dots,\kappa-1\}$ with $a_j \neq 1$. Thus $C_{\nicefrac{1}{\kappa}}$
  is obtained from $[0,1]$ by removing the second $\kappa-$th part of the line segment $[0,
  1]$, and then removing the second interval of length $\nicefrac1{\kappa}$
  from the remaining intervals, and so on. This set
  can be thought of as a directed tree with regular $m-$branching with $m=\kappa-1$.

  For example, if $\kappa=3$, we identify $[0, 1]$ with $\emptyset,$
  the sequence $(\emptyset, 0)$ with the first interval right $[0,
  \nicefrac13]$,
  the sequence $(\emptyset, 1)$ with the central interval $[\nicefrac13,
  \nicefrac23]$ (that is removed),
  the sequence $(\emptyset, 2)$ with the left interval  $[\nicefrac23, 1],$
  the sequence $(\emptyset, 0 ,0)$ with the interval
  $[0, \nicefrac{1}{9}]$ and so on.}
\end{ex}

\subsection{Averaging Operator}
The following definition is taken from \cite{ary}.
Let $$F\colon\R^m\to\R$$ be a continuous function. We call $F$
an averaging operator if it satisfies the following set of conditions:
\begin{enumerate}[(i)]
	\item $F(0,\dots,0)=0$ and $F(1,\dots,1)=1$;
	\item $F(tx_1,\dots,tx_m)=t F(x_1,\dots,x_m)$ for all $t\in\R;$
	\item $F(t+x_1,\dots,t+x_m)=t+ F(x_1,\dots,x_m)$ for all
			$t\in\R;$
	\item $F(x_1,\dots,x_m)<\max\{x_1,\dots,x_m\}$ if not all
		   $x_j$'s are equal;
    \item $F$ is nondecreasing with respect to each variable.
\end{enumerate}

\begin{re}\label{Fmax}
It holds that, if $(x_1,\dots,x_m),
(y_1,\dots,y_m)\in\R^m,$ then
\[
x_j\le y_j + \max_{1\le j\le m}\left\{ x_j-y_j\right\}
\]
for all $j\in\{1,\dots,m\}$. Let $F$ be an averaging operator. Then, by (iii) and (v),
\[
F(x_1,\dots,x_m)\le F(y_1,\dots,y_m)+
\max_{1\le j\le m}\left\{ x_j-y_j\right\}.
\]
Therefore
\[
F(x_1,\dots,x_m)-F(y_1,\dots,y_m)\le
\max_{1\le j\le m}\left\{ x_j-y_j\right\},
\]
and moreover
\[
|F(x_1,\dots,x_m)-F(y_1,\dots,y_m)|\le
\max_{1\le j\le m}\left\{ |x_j-y_j|\right\}.
\]
\end{re}

Now we give some examples.

\begin{ex}
{\rm This example is taken from \cite{KLW}. For $1<p<+\infty,$
	the operator
	$$F^p(x_1,\dots,x_m)=t$$ from $\R^m$ to $\R$  defined
	implicity by
	\[
	\sum_{j=1}^m (x_j-t)|x_j-t|^{p-2}=0
	\]
	is a permutation invariant averaging operator.}
\end{ex}

\begin{ex}
{\rm
For $0\le\alpha,\beta\le1$ with
		$\alpha+\beta=1$, let us consider
		\begin{align*}
			&F_1(x_1,\dots,x_m)=\alpha
			\underset{{1\le j\le m}}{\med}\{x_j\}+
			 \frac{\beta}m\sum_{j=1}^m x_j,\\
			 &F_2(x_1,\dots,x_m)=\alpha
			\underset{{1\le j\le m}}{\med}\{x_j\}+
			 \frac{\beta}2
			 \left(\max_{1\le j\le m}
			\{x_j\}+\min_{1\le j\le m}\{x_j\}
			\right),
		\end{align*}
		where
		\[
\underset{{1\le j\le m}}{\med}\{x_j\}\, :=
   \begin{cases}
  y_{\frac{m+1}2} & \text{ if }m \text{ is even},  \\
  \displaystyle\frac{y_{\frac{m}2}+ y_{(\frac{m}2 +1)}}2
  & \text{ if }m \text{ is odd},
  \end{cases}
\]
with $\{y_1,\dots, y_m\}$ a nondecreasing rearrangement of
$\{x_1,\dots, x_m\}.$

It holds that $F_1$ and $F_2$ are permutation invariant
averaging operators.}
\end{ex}

\section{Existence and Uniqueness}\label{sect-exis-uni}

First we show that there exists a unique
solution of problem
\eqref{CP}
in the space $\Lloc.$


\begin{proof}[Proof of Theorem \ref{TEU}]
{\it Existence.} Let $T>0$ and
\[
C_T:=\{u\in L^{\infty}(\mathbb{T}_m\times[0,T], \mathbb{R}) \colon
u(x,t) \mbox{ is continuous in }t \}.
\]
Observe that $C_T$ is a Banach space with the $L^\infty$-norm.

We can see that $K_f$ is a contraction on $C_T$. In fact,
using Remark \ref{Fmax}, we have that
\[
\|K_fu_1-K_fu_2\|_{\infty}\le\int_{0}^t  e^{z-t} \, dz \|u_1-u_2\|_{\infty} \le (1-e^{-T})\|u_1-u_2\|_{\infty},
\]
for all $u_1, u_2\in C_T.$ Therefore, by the Brouwer fixed-point theorem, $K_f$ has a unique fixed point
$u\in  C_T.$

Since $T>0$ is arbitrary, we can obtain a globally defined solution of \eqref{EI}, $u$.
\medskip

{\it Uniqueness.} Let $u,v$ be two solutions of \eqref{CP} such that
$$u,v\in \Lloc.$$ Then, $u,v$ are solutions of
\eqref{EI} and therefore they are fixed points of $K_f.$ Thus $u
\equiv v$ in $\T\times [0, T]$ for all $T>0$ due to $K_f$ is a
contraction operator. Therefore $u\equiv v$ in
$\T\times[0,+\infty).$
\end{proof}

\begin{re}
We note that there is no need of a ``boundary condition". This problem can be regarded as the analogous for the tree to the Cauchy problem for a PDE, as $u_t = \Delta u$ in $\R^n \times (0, \infty)$ with $u(x,0) =f(x)$ in $\R^n$.
Here we consider $f \in L^\infty$, but the result can be slightly improved to allow for an unbounded initial condition, see Remark \ref{remark.3.3} below.
\end{re}

Next we show a comparison principle.


\begin{proof}[Proof of Theorem \ref{TCOMP}]
Let $T>0.$
We consider
\[
M_T:= \displaystyle\sup_{\T\times [0, T]}\{u-v\}.
\]
Then, given $\varepsilon>0,$ there exists
$(\tilde{x}, \tilde{t})\in \T\times [0, T]$ such that
\[
M_T-\varepsilon\le u(\tilde{x}, \tilde{t})- v(\tilde{x}, \tilde{t}).
\]
Now, by \eqref{des}, we obtain that
\begin{align*}
&M_T-\varepsilon\le u(\tilde{x}, \tilde{t})- v(\tilde{x}, \tilde{t})\\
 & \le \int_{0}^{\tilde{t}}
e^{z-\tilde{t}} \Big(F(u((\tilde{x},0),z),\dots,u((\tilde{x},m-1),z)) \\
& \qquad \qquad \qquad -F(v((\tilde{x},0),z),\dots,v((\tilde{x},m-1),z))\Big) dz\\
&\qquad  + e^{-\tilde{t}}(f(\tilde{x})- g(\tilde{x})).
\end{align*}

Thus, using that $f\le g$ in $\T$ and Remark \ref{Fmax}, we have that
\[
M_T-\varepsilon\le M_T(1-e^{-T}),
\]
and therefore $e^{-T}M_T\le\varepsilon$ for all $\varepsilon>0.$ Then, using that $e^{-T}>0,$ we obtain that $M_T\le 0$
and this implies that $u(x,t)\le v(x,t)$ for all $(x,t)\in \T\times [0, T].$

Since $T>0$ is arbitrary, we can conclude that $u\le v$ in $\T\times [0, +\infty).$
\end{proof}

\begin{co}
Let $F$ be an averaging operator and $f\in L^\infty(\T,\R).$ Then, any bounded solution $u$ of \eqref{CP} with initial condition $f$  satisfies the inequality
$$
|u(x,t)|\le \|f\|_{L^\infty(\T,\R)}
$$
for all $(x,t)\in\T\times [0,+\infty).$
\end{co}
\begin{proof}
We just observe that $w(x,t) =M=\|f\|_{L^\infty(\T,\R)}$ is the solution of \eqref{CP} with initial condition $M.$ Since $f(x) \leq M$, from Theorem \ref{TCOMP}, we obtain that
\[
u(x,t)\leq M, \mbox{ for all } (x,t)\in\T\times [0,+\infty).
\]

In a similar way, we can prove that $u(x,t)\ge -M$ for all $(x,t)\in\T\times [0,+\infty).$ Therefore,
\[
|u(x,t)|\le \|f\|_{L^\infty(\T,\R)} \quad \forall (x,t)\in\T\times [0,+\infty).
\]
This completes the proof.
\end{proof}

\begin{re}\label{remark.3.3}
We remark that we can have existence of a solution even if the initial condition $f$
is not bounded. In fact, we just observe that
\[
u(x,t)=C e^{(\lambda-1)t} \lambda^{\l(x)},
\]
with $\lambda>0$ is a solution of \eqref{CP} with initial condition $f(x)=C \lambda^{\l(x)}.$

Then, there is a solution of \eqref{CP} for any initial condition
such that
\[
0\leq f(x) \leq C \lambda^{\l(x)}.
\]
To obtain such a solution we generate a
sequence of approximating solutions using truncations of the initial condition. In fact,
let
\[
f_n(x)=\min \{f(x), n \},\quad u_n(x,t)=\min \{u(x,t), n \},
\]
and take
$w_n(x,t)\in\Lloc$  the unique solution of \eqref{CP} with initial condition $f_n$
(Theorem \ref{TEU}).

We can see that, $u_n\to u$ as $n\to+\infty,$
$K_{f_n}u_n\le u_n,$ and, by the comparison principle,
$w_n$ is increasing with $n$ and $w_n\le u_n.$

Finally, taking the limit as $n\to \infty$ in the form of the equation given by
\eqref{EI}, we obtain that $w(x,t):=\lim\limits_{n\to+\infty}w_n(x,t)$ is a solution of
\eqref{CP} with initial condition $w(x,0)=f(x)$.
\end{re}

\section{Decay Estimates}\label{sect-decay}

First, we prove Theorem \ref{TEST}.
\begin{proof}[Proof of Theorem \ref{TEST}]
We begin by observing that if $f\equiv0$ on $\T$
then  $u\equiv0$ on $\T\times[0,+\infty).$ Therefore,
\eqref{estimacion.2} holds trivially in this case.

\medskip

Now, we consider the case $f\not\equiv0.$ Then
$\textit{a}(f)\neq0,$ $f(x)\neq 0$ for some $x\in\T^{\mu(f)}$
and $f(x)=0$ for all $x$ such that $\textit{l}(x)> \mu(f).$
Thus, by Theorem \ref{TEU}, $u(x,t)=0$ for all $x$ such that $
\textit{l}(x)>\mu(f).$ Therefore,
if $x\in\T^{\mu(f)},$ we have that
\begin{align*}
	u_t (x,t)&= F(u((x,0), t), \dots, u((x,m-1), t)) -u(x,t) \\[6pt]
                         &= F(0,\dots, 0)-u(x,t)= -u(x,t).
\end{align*}
Then
\[
\frac{d}{dt}\left(e^t u(x,t)\right)=0.
\]
Since $u(x,0)=f(x)$ for all $x\in\T,$ we get
\[
	u(x,t)= f(x)e^{-t} \quad\forall x\in\T^{\mu(f)}.
\]

Thus, for any $x\in\T^{\mu(f)-1}$ we have that
\begin{align*}
	u_t(x,t)&=F(u((x,0), t), \dots, u((x,m-1), t)) -u(x,t)\\
			&=F(f(x,0)e^{-t}, \dots, f(x, m-1)e^{-t}) -u(x,t)\\
			&=F(f(x,0), \dots, f(x, m-1))e^{-t} -u(x,t).
\end{align*}
Then,
\[
	\frac{d}{dt}\left(e^tu(x,t)\right)= \mathcal{A}_x^{1},
\]
where $\mathcal{A}_x^{1}= F(f(x,0), \dots, f(x, m-1)).$ Therefore,
\begin{equation}\label{nivelmu-1}
u(x,t)= (\mathcal{A}_x^{1} t+f(x))e^{-t} \quad \forall
x\in\T^{\mu(f)-1}.
\end{equation}

Observe that
\begin{equation}\label{cotaA1}
	|\mathcal{A}_x^{1}|\le \|f\|_{L^\infty(\T,\R)}
	\quad \forall x\in\T^{\mu(f)-1},
\end{equation}
due to the fact that $F$ is nondecreasing with
respect to each variable.

Arguing as before, using
\eqref{nivelmu-1}, we obtain
\[
	\frac{d}{dt}\left(e^tu(x,t)\right)
	=  F\left(\mathcal{A}_{(x,0)}^{1}t + f(x,0), \dots,
	\mathcal{A}_{(x,m-1)}^{1}t + f(x,m-1)\right),
\]
for every $x\in\T^{\mu(f)-2}$.
Then, since $F$ is nondecreasing with  respect to each variable,
 we have that
\[
	\mathcal{A}_x^{2}t-\|f\|_{L^\infty(\T,\R)}
	\le \frac{d}{dt}\left(e^tu(x,t)\right)\le
	\mathcal{A}_x^{2}t+\|f\|_{L^\infty(\T,\R)}
	\quad \forall x\in\T^{\mu(f)-2},
\]
where
\[
	\mathcal{A}_x^{2}=
	F\left(\mathcal{A}_{(x,0)}^{1}, \dots,
	\mathcal{A}_{(x,m-1)}^{1}\right).
\]
Therefore
\[
\begin{array}{l}
\displaystyle
	e^{-t}\left(\mathcal{A}_x^{2}\frac{t^2}2-
	\|f\|_{L^\infty(\T,\R)}t+f(x)\right)
	\le u(x,t) \\[12pt]
\displaystyle \qquad \qquad \le e^{-t}
	\left(\mathcal{A}_x^{2}\frac{t^2}2+
	\|f\|_{L^\infty(\T,\R)}t+f(x)\right),
\end{array}
\]
for all  $x\in\T^{\mu(f)-2}.$

By \eqref{cotaA1},
using again that $F$ is nondecreasing with respect to each
variable, we obtain
\[
	|\mathcal{A}_x^{2}|\le \|f\|_{L^\infty(\T,\R)} \quad
	\forall  x\in\T^{\mu(f)-2}.
\]

Continuing in the same manner, we can prove
\[
	e^{-t}p_1(t)\le u(\emptyset,t)\le e^{-t}p_2(t),
\]
where
\begin{align*}
	p_1(t)=&\mathcal{A}_{\emptyset}^{\mu(f)}
		\frac{t^{\mu(f)}}{\mu(f)!}- \sum_{j=1}^{\mu(f)-1}
		\frac{t^{j}}{j!}
		\|f\|_{L^\infty(\T,\R)}
		+f(\emptyset),\\[5pt]
	p_2(t)=&\mathcal{A}_{\emptyset}^{\mu(f)}
		\frac{t^{\mu(f)}}{\mu(f)!}+
		\sum_{j=1}^{\mu(f)-1}
		\frac{t^{j}}{j!}
		\|f\|_{L^\infty(\T,\R)}
		+f(\emptyset),\\[5pt]
	\mathcal{A}_{\emptyset}^{\mu(f)}=&
	F\left(\mathcal{A}_{(\emptyset,0)}^{\mu(f)-1},
	\dots, \mathcal{A}_{(\emptyset,m-1)}^{\mu(f)-1}\right).
\end{align*}
Arguing as before, we have that
\[
	|\mathcal{A}_{\emptyset}^{\mu(f)}|\le
	\|f\|_{L^\infty(\T,\R)}.
\]
Thus,
\[
	\max_{x\in\T} |u(x,t)| \le
	\frac{t^{\mu(f)}e^{-t}}{\mu(f)!}\|f\|_{L^\infty(\T,\R)},
\]
for $t$ large enough.
\end{proof}

\begin{re}\label{cota.opt}
The bound that we obtained in Theorem \ref{TEST} is optimal.
In fact, let $n\in\N,$ $F$ be an averaging operator and $f_n\in L^{\infty}(\T, \R)$ defined as
\[
	f_n(x):=
	\begin{cases}
		n! & \mbox{ if } l(x)=n,\\
		0 & \mbox{ if } l(x)\neq n.
	\end{cases}
\]
Note that $\|f_n\|_{L^\infty(\T,\R)}=n!$ and $\mu(f_n)=n.$
Let
\[
	z_n(x,t):=
	\begin{cases}
		\displaystyle \dfrac{n!}{(n-l(x))!}t^{(n-l(x))} & \mbox{ if } 0\le l(x)\le n,\\
		0 & \mbox{ if } l(x)>n.
	\end{cases}
\]
Then, we can observe that $u_n(x,t):= e^{-t}z_n(x,t)\in\Lloc,$ $u_n$ is the solution of \eqref{CP} with initial condition $f_n,$ and
\[
	\max_{x\in\T}|u_n(x,t)|= t^ne^{-t}= \frac{t^{\mu(f_n)}e^{-t}}{\mu(f_n)!}\|f_n\|_{L^\infty(\T,\R)}.
\]
\end{re}

\begin{pro} \label{RiBer.es.el.mejor.del.mundo}
Let $F$ be an averaging operator and $f(x)= (1-\lambda)^{l(x)}$ for some $\lambda\in (0,1).$ Then $u(x,t)= e^{-\lambda t} f(x)$ is the solution to \eqref{CP}.
\end{pro}

\begin{proof}
We have that $u(x,0)= f(x)$ for all $x\in\T$ and
\begin{align*}
	\Delta_F u (x,t)&= F(u((x,0),t),\dots,u((x,m-1),t))-u(x,t)\\
	&=F(e^{-\lambda t}f(x,0), \dots, e^{-\lambda t}f(x,m-1))- e^{-\lambda t}f(x)\\
	&= e^{-\lambda t}F(f(x,0),\dots, f(x,m-1))- e^{-\lambda t}f(x)\\
	&= e^{-\lambda t}F((1-\lambda)^{l(x)+1}, \dots, (1-\lambda)^{l(x)+1}) - e^{-\lambda t}(1-\lambda)^{l(x)}\\
	&= e^{-\lambda t} (1-\lambda)^{l(x)+1}- e^{-\lambda t} (1-\lambda)^{l(x)}\\
	&= e^{-\lambda t} (1-\lambda)^{l(x)}(1-\lambda+1)\\
	&= -\lambda e^{-\lambda t}(1-\lambda)^{l(x)}\\
	&= u_t(x,t),
\end{align*}
for all $(x,t)\in\T\times (0, +\infty).$
\end{proof}

We observe that for this particular solution we have
\[
	\max_{x\in\T} u(x,t)= e^{-\lambda t} \max_{x\in\T}f(x)= e^{-\lambda t}= u(\emptyset, t).
\]
Therefore, using the comparison principle stated in Theorem \ref{TCOMP}, we obtain
Theorem \ref{RiBer.es.la.B} as an immediate consequence.
Proposition \ref{RiBer.es.el.mejor.del.mundo} shows that the bound is optimal.

\medskip

Finally, let us prove that there are solutions with any prescribed behaviour of $u(\emptyset, t)$.

\begin{proof}[Proof of Theorem \ref{todo.vale}]
We just consider $u(x,t) \equiv a_{l(x)} (t)$ (that is, we take $u$ to be constant at every level). Then the equation
reduces to find $a_1, a_2,\dots, a_n,\dots$  such that
\[
a_i^\prime (t) = a_{i+1} (t) - a_i(t),
\]
that is,
\[
a_{i+1}(t) = a_{i}^\prime (t) + a_i(t).
\]
Hence, given $a_0$, we can construct
\begin{align*}
a_1 (t) &= a_0^\prime(t) + a_0(t),\\
a_2(t) &= a_0^{\prime\prime}(t) + 2 a_0^\prime (t) + a_0(t),
\end{align*}
etc, that is, at level $n$, we have
\[
a_n (t) = \sum_{j=0}^n \binom{n}{j} a_0^{(j)} (t).
\]
Therefore
\[
u(x,t)= \sum_{j=0}^{l(x)}  \binom{l(x)}{j} a_0^{(j)} (t)
\]
is a solution of the equation.
\end{proof}

Remark that depending on the behaviour of the derivatives of $a_0$ it may hold that
$$u(\emptyset,t)=a_0(t) =	\max_{x\in\T} u(x,t). $$

If we have
\[
a_0 (t) = (1+t)^{-\alpha}\quad (\alpha>0)
\]
then we get
\[
u(x,t) = \sum_{j=0}^{l(x)} \binom{l(x)}{j}
a_0^{(j)} (t) = \sum_{j=0}^{l(x)}  \binom{l(x)}{j} (-1)^j \left(\prod_{i=0}^{j-1} (\alpha +i)\right) (1+t)^{-\alpha-j}.
\]
Note that we have as initial condition for this particular solution
\[
f(x)  = \sum_{j=0}^{l(x)}  \binom{l(x)}{j} (-1)^j \left(\prod_{i=0}^{j-1} (\alpha +i)\right) .
\]

Note that this initial condition can be unbounded. For example, for $\alpha =1$ we have
\[
a_0 (t) = (1+t)^{-1}.
\]
Then we get
\begin{align*}
u(x,t)&= \sum_{j=0}^{l(x)}  \binom{l(x)}{j}
a_0^{(j)} (t)\\
&= \sum_{j=0}^{l(x)}  (-1)^j \binom{l(x)}{j}j!(1+t)^{-(j+1)}\\
&=\sum_{j=0}^{l(x)}  (-1)^j \dfrac{l(x)!}{(l(x)-j)!}(1+t)^{-(j+1)}
\end{align*}
is a solution of \eqref{CP} with initial condition
\[
f(x) = \sum_{j=0}^{l(x)}   (-1)^j \dfrac{l(x)!}{(l(x)-j)!}\,
=(-1)^{l(x)}!l(x),
\]
where $!n$ denotes the subfactorial of $n.$

We can observe that $f(x)$ is an oscillating function with
$$|f(x)|\to+\infty\qquad \mbox{ as }\qquad  l(x)\to+\infty.$$


\end{document}